\documentclass[preprint,12pt]{elsarticle}
\usepackage{amssymb}
\usepackage{amsmath}
 \usepackage{amsthm}
 \newtheorem{lemma}{Lemma}[section]
 \newtheorem{theorem}{Theorem}[section]
 
\newtheorem{corollary}{Corollary}[section]
\journal{Discrete Appl. Math.}
\begin{document}
\begin{frontmatter}

\title{Equitable coloring of Kronecker products of complete multipartite graphs and complete graphs}
\author{Zhidan Yan }
\author{Wei Wang\corref{cor1}}
\cortext[cor1]{Corresponding author.
Fax:+86-997-4682766.\\
\ead{wangwei.math@gmail.com}}

\address{College of Information Engineering, Tarim University, Alar 843300, China}

\begin{abstract}
A proper vertex coloring of a graph is equitable if the sizes of
color classes differ by at most 1. The equitable chromatic number of
a graph $G$, denoted by $\chi_=(G)$, is the minimum $k$ such that
$G$ is equitably $k$-colorable. The equitable chromatic threshold of
a graph $G$, denoted by $\chi_=^*(G)$, is the minimum $t$ such that
$G$ is equitably $k$-colorable for $k \ge t$. In this paper, we give
the exact values of $\chi_=(K_{m_1, \cdots, m_r} \times K_n)$
and $\chi_=^*(K_{m_1,\cdots, m_r} \times K_n)$ for $\sum_{i =
1}^r m_i \leq n$.
\end{abstract}


\end{frontmatter}

\section{Introduction}
\label{intro} All graphs considered in this paper are finite, undirected and without loops or multiple edges.
For a positive integer $k$, let $[k] = \{1,2,\cdots,k\}$. A (proper) $k$-coloring of a graph $G$ is a mapping $f
: V(G) \rightarrow [k]$ such that $f(x) \neq f(y)$ whenever $xy \in E(G)$. We call the set $f^{-1}(i)= \{x \in
V(G)\colon\,f(x) = i\}$ a color class for each $i \in [k]$. A graph is $k$-colorable if it has a $k$-coloring.
The chromatic number of $G$, denoted by $\chi(G)$, is equal to min\{$k\colon\,G$ is $k$-colorable\}. An
equitable $k$-coloring of $G$ is a $k$-coloring for which any two color classes differ in size by at most 1, or
equivalently, each color class is of size $\lfloor|V(G)|/k\rfloor$ or $\lceil|V(G)|/k\rceil$. The equitable
chromatic number of $G$, denoted by $\chi_ = (G)$, is equal to min \{$k\colon\,G$ is equitably $k$-colorable \},
and the equitable chromatic threshold of a graph $G$, denoted by $\chi_=^*(G)$, is equal to min \{$t\colon\, G$
is equitably $k$-colorable for $k\geq t$\}. The Kronecker (or cross, direct, tensor, weak tensor or categorical)
product of graphs $G$ and $H$ is the graph $G \times H$ with vertex set $V(G) \times V(H)$ and edge set $\{(x,
y)(x^\prime, y^\prime)\colon\,xx^\prime \in E(G), yy^\prime \in E(H)\}$.

The concept of equitable colorability was first introduced by Meyer \cite{meyer1973}. The definitive survey of
the subject was given by Lih \cite{lih1998}. For its many application such as scheduling and constructing
timetables, please see \cite{baker1996,janson2002,kitagawa1988,pelsmajer2004,smith1996,tucker1973}.

In 1964, Erd\H{o}s \cite{erdos1964} conjectured that any graph $G$ with maximum degree $\Delta(G)\le k$ has an
equitable $(k + 1)$-coloring, or equivalently, $\chi_=^*(G)\leq \Delta(G) + 1$. This conjecture was proved  in
1970 by Hajnal and Szemer\'{e}di \cite{hajnal1970}. Recently, Kierstead and Kostochka \cite{kierstead2008} gave
a short proof of the theorem, and presented a polynomial algorithm for such a coloring.  Brooks' type results
are conjectured: Equitable Coloring   Conjecture \cite{meyer1973} $\chi_=(G) \leq \Delta(G)$, and Equitable
$\Delta$-Coloring Conjecture \cite{chen1994b} $\chi_=^*(G)\leq\Delta(G)$ for $G\notin\{K_n, C_{2n+1},
K_{2n+1,2n+1}\}$. Equitable coloring has been extensively studied, please see
\cite{chen1994a,chen1994b,chen2009,duffus1985,kostochka2002}. Exact values of equitable chromatic numbers of
trees \cite{chen1994a} and complete multipartite graphs \cite{blum2003,lam2001,lin2010} were determined. Among
the known results, we are most interested in those on graph products, see
\cite{chen2009,furmanzyk2006,lin2010,lin2012,zhu1998}. Duffus, Sands and Woodrow \cite{duffus1985} showed that
$\chi(K_m \times K_n)=\min\{\chi (K_m),\chi(K_n)\}= \min\{ m, n\}$, and Chen, Lih, and Yan \cite{chen2009} got
that $\chi_=(K_m \times K_n) = \min\{m,n\}$. Recently, among other results, Lin and Chang~\cite{lin2010}
established an upper bound on $\chi_=^*(K_m \times K_n)$(see Lemma~\ref{upperbound} below). They also determined
exact values of $\chi_=^*(G \times K_n)$ when $G$ is $P_2$, $P_3$, $C_3$ or $C_4$. The exact value of
$\chi_=^*(K_m \times K_n)$ was obtained by Yan and Wang \cite{yan2012b}. The aim of the present paper is to
determine $\chi_=(K_{m_1, \cdots, m_r} \times K_n)$ and $\chi_=^*(K_{m_1, \cdots, m_r} \times K_n)$ for $\sum_{i
= 1}^r m_i \leq n$.

\section{Preliminaries}

Before stating our main result, we need several preliminary results on integer partitions. Recall that a
partition of an integer $n$ is a sum of the form $n = t_1 + t_2 + \cdots + t_k$, where $0\leq t_i \leq n$ for
$i\in [k]$. We call such a partition a $q$-partition if each $t_i$ is in the set $\{q, q+1\}$. A $q$-partition
of $n$ is typically denoted as $n = aq + b(q + 1)$, where $n$ is the sum of $a$ $q$'s and $b$ $(q + 1)$'s. A
$q$-partition of $n$ is called a minimal(respectively, maximal) $q$-partition if the number of its addends, $a +
b$, is as small(respectively, large) as possible. For example, $8=2 + 2 + 2 + 2$ is a maximal $2$-partition of
$8$, and $8=2 + 3 + 3$ is a minimal $2$-partition of $8$.

Our first lemma is from \cite{blum2003}, which characterizes the condition when a $q$-partition of $n$ exists.
For the sake of completeness, here we restate their proof. In what follows, all variables are nonnegative
integers.

 \begin{lemma}\label{basic}\cite{blum2003}
 If $0 < q \leq n$, and $n=kq+r$ with $0 \leq r < q$, then there is
 a $q$-partition of $n$ if and only if $r\leq k $.
 \end{lemma}

 \begin{proof}
 If $r\leq k $, then $n = (k - r)q + r(q + 1)$ is a $q$-partition of $n$.
 Conversely, given a $q$-partition $n = aq + b(q + 1)$ of $n$, we have
 $n = (a + b)q + b$, so $(a + b) \leq k$ and $r \leq b$. Consequently, $r \leq b \leq (a + b) \leq
  k$.
 \end{proof}

\begin{corollary}\label{noq-partition}\cite{yan2012b}
There is no $q$-partition of $n$ if and only if $\lceil n/(q + 1) \rceil> n/q $.
\end{corollary}

\begin{lemma}\label{minimal1}\cite{blum2003}
A $q$-partition $n = aq + b(q + 1)$ of $n$ is minimal if and only if $a < q + 1$. Moreover a minimal
 $q$-partition is unique.
 \end{lemma}

\begin{lemma}\label{maximal1}\cite{yan2012a}
 A $q$-partition $n = aq + b(q + 1)$ of $n$ is maximal if and only if $b < q$. Moreover a maximal
 $q$-partition is unique.
 \end{lemma}

\begin{lemma}\label{minandmax1}\cite{yan2012a}
If $n= aq + b(q + 1)$ is a minimal $q$-partition, then $a + b = \lceil n/(q + 1)\rceil$. If $n = a^\prime q +
b^\prime(q + 1)$ is a maximal $q$-partition, then $a^\prime + b^\prime = \lfloor n/q \rfloor$. Moreover, when
$\lceil n/(q + 1)\rceil = \lfloor n/q \rfloor$, there is only one $q$-partition of $n$.
\end{lemma}

\begin{lemma}\label{minandmax2}\cite{yan2012a}
Let $n = aq + b(q + 1)$ be the maximal $q$-partition, and $n = a^\prime (q-1) + b^\prime q$ be the minimal $(q -
1)$-partition. If $q|n$ then $a + b = a^\prime + b^\prime$, otherwise, $a + b + 1 = a^\prime + b^\prime$.
\end{lemma}

Denote the partite sets of the graph $K_{m_1n, \cdots, m_rn}$ as $N_i$, with $|N_i|=m_in$ and $i \in [r]$. Any
given color class of an equitable coloring must lie entirely in some $N_i$, for otherwise two of its vertices
are adjacent. Thus, any equitable coloring partitions each $N_i$ into color classes $V_{i_1}, V_{i_2}, \cdots,
V_{i_{v_i}}$, no two of which differ in size by more than 1. If the sizes of the color classes are in the set
$\{q,q+1\}$, then these sizes induce $q$-partitions of each $m_in$. Conversely, given a number $q$, and
$q$-partitions $m_in = aq + b(q + 1)$ of each $m_in$, there is an equitable coloring of $K_{m_1n, \cdots, m_rn}$
with color sizes $q$ and $q+1$; just partition each $N_i$ into $ a_i$ sets of size $q$, and $ b_i$ sets of
$q+1$. It follows, then, that finding an equitable coloring of $K_{m_1n, \cdots, m_rn}$ amounts to finding a
number $q$, and simultaneous $q$-partitions of each of the numbers $m_in$.

\begin{lemma}\label{upperbound}\cite{lin2010}
$\chi_=^*(K_m \times K_n) \leq \lceil {mn}/(m+1) \rceil$ for $m \leq n$.
\end{lemma}

\begin{lemma}\label{completegraph2}\cite{lih1998}
 $\chi_=(K_{n_1, \cdots,
n_t}) = \sum_{i = 1}^t \lceil n_i/h \rceil$, where $h=\textup{max}\{k\colon\,n_i/(k - 1) \geq \lceil n_i/k
\rceil$ for all $i$\}.
\end{lemma}

\begin{corollary}\label{completegraph3}
$\chi_=(K_{m_1n, \cdots, m_rn}) = \sum_{i = 1}^r \lceil m_in/h \rceil$, where $h=\textup{max}\{k\colon\,m_in/(k
- 1) \geq \lceil m_in/k \rceil$ for all $i$\}.
\end{corollary}

\section{The results}
\begin{lemma}\label{compare}
If $K_{m_1, \cdots, m_r}\times K_n$ is equitably $k$-colorable for some $k<\lceil{mn}/(m+1)\rceil$, then
$K_{m_1n, \cdots, m_rn}$ is also equitably $k$-colorable, where $m=\sum_{i=1}^{r}m_i$.
\end{lemma}
\begin{proof}
Let $V(K_{m_1, \cdots, m_r}\times K_n)=\{(x_i^j,y^s)\colon\,i\in[r],j\in[m_i],s\in[n]\}$ and $f$ be an equitable
$k$-coloring of $K_{m_1, \cdots, m_r}\times K_n$ for some $k<\lceil{mn}/(m+1)\rceil$. Then each color class $C$
has size at least $m+1$. We claim that $C$ is a subset of $\{(x_i^j,y^s)\colon\,j\in[m_i],s\in[n]\}$ for some
$i\in[r]$, which implies that $K_{m_1n, \cdots, m_rn}$ is also equitably $k$-colorable.

To show the claim, assume to the contrary that there are $i$ and $i^\prime$, $i\neq i^\prime$, such that
$(x_i^j,y^s),(x_{i^\prime}^{j^\prime},y^{s^\prime})\in C$ for some $j\in[m_i]$, $j^\prime\in[m_{i^\prime}]$,
$s,s^\prime\in[n]$. If $s\neq s^\prime$ then $(x_i^j,y^s)$ is adjacent to
$(x_{i^\prime}^{j^\prime},y^{s^\prime})$, contrary to the fact that $C$ is an independent set. Therefore all
vertices in $C$ take the same value $y^s$ for the second coordinate and hence $|C|\le \sum_{i=1}^{r}m_i$, a
contradiction. The claim follows.
\end{proof}

\begin{lemma}\label{partition}
If $f$ is an equitable $\lfloor{mn}/{(m+1)}\rfloor$-coloring of $K_{m_1n, \cdots, m_rn}$, then each color class
is of size $m+1$ or $m+2$, where $m=\sum_{i=1}^{r}m_i\le n$.
\end{lemma}
\begin{proof}
Let $k=\lfloor{mn}/{(m+1)}\rfloor$. Since each color class is of size $\lfloor{mn}/{k}\rfloor$ or
$\lceil{mn}/{k}\rceil$, it suffices to show that $m+1\le{mn}/{k}\le m+2$.

It is clear that $mn/k\ge m+1$. To show ${mn}/{k}\le m+2$, we consider the following three cases.

{\em Case 1.} $n=m$. We have
$$\frac{mn}{k}=\frac{m^2}{\lfloor\frac{m^2}{m+1}\rfloor}=\frac{m^2}{m-1}\le m+2.$$

{\em Case 2.} $n=m+1$. We have $k=m$ and hence $mn/k=n\le m+2$.

{\em Case 3.} $n\ge m+2$. Write $mn=p(m+1)+s$ with $0\le s<m+1$. We have
$$k=p=\frac{mn-s}{m+1}\ge \frac{mn-m}{m+1}$$ and hence
\begin{eqnarray*}
 \frac{mn}{k}&\le &\frac{mn}{\frac{mn-m}{m+1}}\\
       &=&\frac{n(m+1)}{n-1}\\
       &=&m+1+\frac{m+1}{n-1}\\
       &\le&m+2.
\end{eqnarray*}
\end{proof}
\begin{theorem}\label{theorem1}
$\chi_=(K_{m_1, m_2, \cdots, m_r} \times K_n) = \sum_{i = 1}^r \lceil m_in/h \rceil$ for $\sum_{i = 1}^r m_i\le
n$, where $h=\textup{max}\{k\colon\,m_in/(k - 1) \geq \lceil m_in/k \rceil$ for all $i\}$.
\end{theorem}

\begin{proof}

 Since $K_{m_1, \cdots, m_r} \times K_n$ is a span subgraph of $K_{m_1n, \cdots, m_rn}$,
 Corollary~\ref{completegraph3} implies
  $\chi_=(K_{m_1, \cdots, m_r} \times K_n) \leq\sum_{i=1}^r \lceil m_in/h \rceil$.

Let $m=\sum_{i=1}^{r}m_i$ and $L=\sum_{i=1}^r \lceil m_in/h \rceil-1$. By the definition of $h$, we have $h\ge
n$ and hence $L<m\le\lceil mn/(m+1)\rceil$, where the last inequality follows from the fact $m\le n$ . By
Corollary ~\ref{completegraph3}, $K_{m_1n, \cdots, m_rn}$ is not equitably $L$-colorable. Lemma~\ref{compare}
implies that $K_{m_1, \cdots, m_r} \times K_n$ is also not equitably $L$-colorable.
\end{proof}

\begin{theorem}\label{theorem2}

\begin{equation*}
\chi_=^*(K_{m_1, \cdots, m_r} \times K_n)=\begin{cases}\lceil {mn}/{(m+1)} \rceil
                                  & if~ \sum_{i = 1}^r \lfloor {m_in}/{(m + 1)}\rfloor < \lfloor {mn}/{(m + 1)}\rfloor,\\
                                  &~or~there~is~some~i~such~that~\\
                                  & m_in/(m + 1) < \lceil m_in/(m + 2) \rceil,\\
\sum_{i = 1}^r \lceil m_in/h\rceil&otherwise,
\end{cases}
\end{equation*}
for $m = \sum_{i = 1}^r m_i\le n$, where $h=$\textup{min} \{$t\ge m+2\colon\,$ there is some $i$ such that
$m_in/t < \lceil m_in/(t + 1) \rceil$ or there are $m_i$ and $m_j$, $i \neq j$, such that $t$ divides neither
$m_in$ nor $m_jn$\}.
\end{theorem}

\begin{proof}
Since $K_{m_1, \cdots, m_r}\times K_n$ is a span subgraph of $K_{m}\times K_n$, Lemma~\ref{upperbound} implies
$\chi_=^*(K_{m_1, \cdots, m_r} \times K_n)\leq \lceil {mn}/{(m+1)} \rceil$.

\textit{Case} 1. $\sum_{i = 1}^r \lfloor {m_in}/{(m + 1)}\rfloor < \lfloor {mn}/{(m + 1)}\rfloor$, or there is
some $i$ such that $m_in/(m + 1) < \lceil m_in/(m + 2) \rceil$.

Clearly, $m+1\nmid n$ by the condition of Case 1. Since $\chi_=^*(K_{m_1, \cdots, m_r} \times K_n)\leq \lceil
{mn}/{(m+1)} \rceil$, it suffices to show that $K_{m_1, \cdots, m_r} \times K_n$ is not equitably
$\lfloor{mn}/{(m+1)} \rfloor$-colorable. Suppose to the contrary that $K_{m_1, \cdots, m_r} \times K_n$ is
equitably $\lfloor{mn}/{(m+1)} \rfloor$-colorable. By Lemma~\ref{compare}, $K_{m_1n, \cdots, m_rn}$ is also
equitably $\lfloor{mn}/{(m+1)} \rfloor$-colorable. Let $f$ be any equitable $\lfloor{mn}/{(m+1)}
\rfloor$-coloring of $K_{m_1n, \cdots, m_rn}$. By Lemma~\ref{partition}, each color class of $f$ has size $m+1$
or $m+2$, i.e., $f$ corresponds to simultaneous $(m+1)$-partitions of each of the numbers $m_in$. By
Corollary~\ref{noq-partition}, $m_in/(m+1)\ge\lceil m_in/(m + 2)\rceil$ for $i\in[r]$, which implies $\sum_{i =
1}^r \lfloor {m_in}/{(m + 1)}\rfloor < \lfloor {mn}/{(m + 1)}\rfloor$ by the condition of Case 1. By Lemma
\ref{minandmax1}, any $(m+1)$-partition of $m_in$ contains at most $\lfloor m_in/(m+1)\rfloor$ attends. Since
$f$ corresponds to simultaneous $(m+1)$-partitions of each of the numbers $m_in$ with $\lfloor{mn}/{(m+1)}
\rfloor$ attends all together, we have $\lfloor{mn}/{(m+1)} \rfloor\le\sum_{i = 1}^r \lfloor {m_in}/{(m +
1)}\rfloor$, a contradiction. This proves the theorem for this case.

\textit{Case} 2. $\sum_{i = 1}^r \lfloor {m_in}/{(m +1)}\rfloor \ge \lfloor {mn}/{(m + 1)}\rfloor$ and $m_in/(m+
1) \ge \lceil m_in/(m + 2) \rceil$ for all $i$.

\noindent\textbf{Claim 1.} $K_{m_1n, \cdots, m_rn}$ is equitably $k$-colorable for $\lceil
{mn}/{(m+1)}\rceil-1\ge k\ge\sum_{i = 1}^r \lceil {m_in}/{h}\rceil$.

First we prove that Claim 1 holds for $k=\sum_{i = 1}^r \lceil {m_in}/{h}\rceil$. Set $h^\prime=h-1$. If $h>m+2$
then the definition of $h$ implies that $m_in/h^\prime\ge\lceil m_in/(h^\prime+1)\rceil$ for all $i$, otherwise,
$h^\prime=h-1=m+1$ and the same conclusion also holds by the condition of Case 2.

By Corollary~\ref{noq-partition}, each $m_in$ has an $h^\prime$-partition. By Lemma~\ref{minandmax1}, a minimal
partition of each $m_in$ has $\lceil m_in/h\rceil$ attends, which implies that $K_{m_1n, \cdots, m_rn}$ is
equitably $\sum_{i = 1}^r \lceil {m_in}/{h}\rceil$-colorable. This proves that Claim 1 holds for $k=\sum_{i =
1}^r \lceil {m_in}/{h}\rceil$.

Now assume that Claim 1 is true for some $k$ satisfying
 \begin{equation}\label{assk}
 \lceil\frac{mn}{m+1}\rceil-2\ge k\ge\sum_{i = 1}^r \lceil \frac{m_in}{h}\rceil,
  \end{equation}
and we prove that it is true for $k+1$. Let $m_in=a_iq+b_i(q+1)$ for each $i$ such that
$\sum_{i=1}^{r}(a_i+b_i)=k$. Lemma~\ref{minandmax1} implies
 \begin{equation}\label{resk}
\sum_{i=1}^{r}\lfloor\frac{m_in}{q}\rfloor\ge k\ge \sum_{i=1}^{r}\lceil\frac{m_in}{q+1}\rceil.
  \end{equation}

\noindent\textbf{Claim 1.1.} $m<q<h.$

If $q\le m$ then by (\ref{resk}) and the condition of Case 2 we have
\begin{displaymath}
      k\ge\sum_{i=1}^{r}\lceil\frac{m_in}{q+1}\rceil\ge\sum_{i=1}^{r}\lceil\frac{m_in}{m+1}\rceil\ge\lceil\frac{mn}{m+1}\rceil,
\end{displaymath}
a contradiction to (\ref{assk}).

If $q\ge h$ then by (\ref{resk}) and (\ref{assk}) we have
\begin{displaymath}
\sum_{i=1}^{r}\lfloor\frac{m_in}{h}\rfloor\ge\sum_{i=1}^{r}\lfloor\frac{m_in}{q}\rfloor\ge k\ge\sum_{i = 1}^r
\lceil \frac{m_in}{h}\rceil,
\end{displaymath}
which implies $h\mid m_in$ for all $i$, contrary to the definition of $h$. This proves Claim 1.1.

If there is some $m_i$ whose $q$-partition $m_in=a_iq+b_i(q+1)$ is not maximal, then Lemma~\ref{maximal1}
implies $b_i\ge q$. By using a new partition $m_in=(a_i+q+1)q+(b_i-q)(q+1)$ one finds that Claim 1 is true for
$k+1$.

Now we assume that each $q$-partition $m_in=a_iq+b_i(q+1)$ is maximal. Then Lemma~\ref{minandmax1} implies
\begin{equation}\label{nresk}
k=\sum_{i=1}^{r}(a_i+b_i)=\sum_{i=1}^{r}\lfloor\frac{m_in}{q}\rfloor.
\end{equation}

\noindent\textbf{Claim 1.2.} $m+2\le q<h.$

By Claim 1.1, it suffices to show $q\neq m+1$. Suppose to the contrary that $q=m+1$. Then by (\ref{nresk}) and
the condition of Case 2 we have
$$k=\sum_{i=1}^{r}\lfloor\frac{m_in}{m+1}\rfloor\ge\lfloor\frac{mn}{m+1}\rfloor,$$
a contradiction to (\ref{assk}).

\noindent\textbf{Claim 1.3.} Each $m_in$ has a $(q-1)$-partition.

By Corollary~\ref{noq-partition}, it suffices to show $m_in/(q-1) \ge \lceil m_in/q \rceil$ for all $i$. By
Claim 1.2, either $q-1=m+1$ or $q-1\ge m+2$. It holds when $q-1=m+1$ by the condition of Case 2. If $q-1\ge m+2$
then the definition of $h$ implies the same conclusion.

\noindent\textbf{Claim 1.4.} There is some $i$ such that $\lceil{m_in}/{q}\rceil<\lfloor{m_in}/{(q-1)}\rfloor.$

By Claim 1.2, the definition of $h$ implies that $q\mid m_in$ for all $i$ with at most one exception. Therefore,
$\sum_{i=1}^{r}\lfloor m_in/q\rfloor\ge\sum_{i=1}^{r}\lceil m_in/q\rceil-1.$

Suppose to the contrary that $\lceil{m_in}/{q}\rceil\ge\lfloor{m_in}/{(q-1)}\rfloor$ for all $i$. Note $q-1\ge
m+1$. Combining the three inequalities and the condition of Case 2, from (\ref{nresk}) we have
\begin{eqnarray*}
       k&=&\sum_{i=1}^{r}\lfloor\frac{m_in}{q}\rfloor\\
        &\ge&\sum_{i=1}^{r}\lceil\frac{ m_in}{q}\rceil-1\\
        &\ge&\sum_{i=1}^{r}\lfloor\frac{m_in}{q-1}\rfloor-1\\
        &\ge&\sum_{i=1}^{r}\lfloor\frac{m_in}{m+1}\rfloor-1\\
        &\ge&\lfloor\frac{mn}{m+1}\rfloor-1,
\end{eqnarray*}
a contradiction to (\ref{assk}). Claim 1.4 follows.

Now we can prove that Claim 1 holds for $k+1$ by considering the following two cases.

\textit{Subcase} 2.1. $q\mid m_in$ for all $i$. By Lemma~\ref{minandmax2}, each maximal $q$-partition
$m_in=a_iq+b_i(q+1)$ is a minimal $(q-1)$-partition(i.e., $b_i=0$). By Claim 1.4, there is some $i$ such that
$\lceil{m_in}/{q}\rceil<\lfloor{m_in}/{(q-1)}\rfloor$. Therefore the $(q-1)$-partition of $m_in$ is not maximal
by Lemma~\ref{minandmax1}. Hence, Lemma~\ref{maximal1} implies $a_i\ge q-1$. By using a new partition
$m_in=q(q-1)+(a_i-q+1)q$, one finds that Claim 1 is true for $k+1$.

\textit{Subcase} 2.2. $q\nmid m_in$ for some $i$ and $q\mid m_jn$ for $j\neq i$. As in Subcase 2.1, each maximal
$q$-partition $m_jn=a_jq+b_j(q+1)$ for $j\neq i$ is a minimal $(q-1)$-partition(i.e., $b_j=0$). By Claim 1.3,
$m_in$ has a $(q-1)$-partition. Let $m_in=a^\prime(q-1)+b^\prime q$ be the minimum $(q-1)$-partition of $m_in$.
Lemma~\ref{minandmax2} implies $a^\prime+b^\prime=a_i+b_i+1$. As in Subcase 2.1, one finds that Claim 1 is true
for $k+1$ by using the new partition of $m_in$.

\noindent\textbf{Claim 2.} $K_{m_1, \cdots, m_r} \times K_n$ is equitably $k$-colorable for $k\ge\sum_{i = 1}^r
\lceil {m_in}/{h}\rceil$.

Since $K_{m_1, \cdots, m_r} \times K_n$ is a span subgraph of $K_m\times K_n$, Lemma~\ref{upperbound} implies
that Claim 2 holds for $k\ge \lceil {mn}/{(m+1)}\rceil$. Since $K_{m_1, \cdots, m_r} \times K_n$ is a span
subgraph of $K_{m_1n, \cdots, m_rn}$, Claim 2 holds for $\lceil {mn}/{(m+1)}\rceil-1\ge k\ge\sum_{i = 1}^r
\lceil {m_in}/{h}\rceil$ by Claim 1. This proves Claim 2.

\noindent\textbf{Claim 3.} $K_{m_1n, \cdots, m_rn}$ is not equitably $k$-colorable for $k=\sum_{i = 1}^r \lceil
{m_in}/{h}\rceil-1$.

Suppose to the contrary that $K_{m_1n, \cdots, m_rn}$ is equitably $(\sum_{i = 1}^r \lceil m_in/h \rceil -
1)$-colorable. Then, each $m_in$ has a $q$-partition $m_in = a_iq + b_i(q + 1)$ such that $k = \sum_{i = 1}^r
(a_i + b_i) = \sum_{i = 1}^r \lceil m_in/h \rceil - 1$. Lemma~\ref{minandmax1} implies
\begin{equation}\label{A}
\sum_{i=1}^{r}\lceil\frac{m_in}{q+1}\rceil\leq\sum_{i = 1}^r \lceil\frac{ m_in}{h} \rceil -
1\leq\sum_{i=1}^{r}\lfloor\frac{m_in}{q}\rfloor,
\end{equation}
and hence $q\ge h$ by the first inequality.

Now we show that either of the following two cases will yield a contradiction.

\textit{Subcase} $2.1^\prime.$ There are $m_i$ and $m_j$, $i \neq j$, such that $h$ divides neither $m_in$ nor
$m_jn$. Since $q\ge h$ , we have
$$\sum_{i=1}^{r}\lfloor\frac{m_in}{q}\rfloor\leq\sum_{i = 1}^r \lfloor\frac{
m_in}{h} \rfloor\leq\sum_{i=1}^{r}\lceil\frac{m_in}{h}\rceil-2,$$ a contradiction to (\ref{A}).

\textit{Subcase} $2.2^\prime.$ There is some $i$ such that $m_in/h < \lceil m_in/(h+ 1) \rceil$. By
Corollary~\ref{noq-partition}, $m_in$ has no $h$-partition, yielding $q\neq h$ and hence $q\ge h+1$. We have
\begin{displaymath}
       \sum_{i=1}^{r}\lfloor\frac{m_in}{q}\rfloor\le\sum_{i=1}^{r}\lfloor\frac{m_in}{h+1}\rfloor
                                                  \le\sum_{j=1}^{r}\Bigl(\lceil\frac{m_in}{h}\rceil-1\Bigr)
                                                 \le\sum_{i = 1}^r \lceil\frac{m_in}{h} \rceil - 2,
\end{displaymath}
a contradiction to (\ref{A}).

\noindent\textbf{Claim 4.} $K_{m_1, \cdots, m_r} \times K_n$ is not equitably $k$-colorable for $k=\sum_{i =
1}^r \lceil {m_in}/{h}\rceil-1$.

By Lemma~\ref{compare} and Claim 3, it suffices to show $\sum_{i = 1}^r \lceil {m_in}/{h}\rceil-1<\lceil
mn/(m+1)\rceil$. It is clear that $h\ge m+2$ by the definition of $h$. Recall $m_in/(m + 1) \ge \lceil m_in/(m +
2) \rceil$ for all $i$. We have
\begin{eqnarray*}
       \sum_{i = 1}^r \lceil \frac{m_in}{h}\rceil-1&<&\sum_{i = 1}^r \lceil \frac{m_in}{h}\rceil~~~~~~~\\
                                                 &\le&\sum_{i = 1}^r\lceil\frac{m_in}{m+2}\rceil\\
                                                 &\le&\sum_{i = 1}^r\frac{m_in}{m+1}\\
                                                 &=&\frac{mn}{m+1}\\
                                                 &\leq&\lceil\frac{mn}{m+1}\rceil,
\end{eqnarray*}
as desired. The proof of the theorem in this case is complete by Claims 2 and 4.
\end{proof}

\end{document}